\documentclass[12pt]{article}
\usepackage{amsmath,mdframed}
\usepackage{amssymb}
\usepackage{latexsym}
\usepackage{amsthm}
\usepackage{mathrsfs}
\usepackage{enumitem}
\usepackage[colorlinks,
            linkcolor=red,
            anchorcolor=blue,
            citecolor=green
            ]{hyperref}
\usepackage{graphicx}
\usepackage{diagbox}
\usepackage{tikz}
\usepackage{tkz-graph}
\usetikzlibrary{graphs}
\usetikzlibrary{graphs.standard}
\usetikzlibrary{arrows,decorations.markings}
\tikzstyle{pointV}=[circle,fill=black,inner sep=0.5mm]

\newtheorem{thm}{Theorem}[section]
\newtheorem{lem}[thm]{Lemma}

\newtheorem{rem}[thm]{Remark}
\newtheorem{deff}[thm]{Definition}

\DeclareMathOperator{\red}{red}

\makeatother

\begin{document}

\begin{center}
{\large \bf  On a family of universal cycles for multi-dimensional permutations}
\end{center}

\begin{center}
Sergey Kitaev$^{1}$ and Dun Qiu$^{2}$\\[6pt]

$^{1}$ Department of Mathematics and Statistics \\
University of Strathclyde, 26 Richmond Street, Glasgow G1 1XH, UK\\[6pt]

$^{2}$Center for Combinatorics, LPMC, Nankai University, Tianjin 300071, P. R. China \\[6pt]

Email: $^{1}${\tt sergey.kitaev@strath.ac.uk},
           $^{2}${\tt qiudun@nankai.edu.cn}
\end{center}

\noindent\textbf{Abstract.} A universal cycle (u-cycle) for permutations of length $n$ is a cyclic word, any size $n$ window of which is order-isomorphic to exactly one permutation of length $n$, and all permutations of length $n$ are covered. It is known that u-cycles for permutations exist, and they have been considered in the literature in several papers from different points of view.   

In this paper, we show how to construct a family of u-cycles for multi-dimensional permutations, which is based on applying an appropriate greedy algorithm. Our construction is a generalisation of the greedy way by Gao et al.\ to construct u-cycles for permutations. We also note the existence of u-cycles for $d$-dimensional matrices.\\



\section{Introduction}\label{intro}

A {\em universal cycle}, or {\em u-cycle},  for a given set $S$ with $\ell$ words of length $n$ over an alphabet $A$ is a circular word $u_0u_1\cdots u_{\ell-1}$ that contains each word from $S$ exactly once (and no other word) as a factor  $u_iu_{i+1}\cdots u_{i+n-1}$ for some $0\leq i\leq \ell-1$, where the indices are taken modulo $\ell$. For example, 020311 is a u-cycle for the set of words $\{020,031,102,110,203,311\}$. The notion of a universal cycle for combinatorial structures was introduced in \cite{CDG1992} and has been studied in the literature extensively for various objects. The celebrated {\em de  Bruijn sequences} are a particular case of such a u-cycle, where a set in question is the set $A^n$ of all words of length $n$ over a $k$-letter alphabet $A$. De Bruijn sequences of orders $2$ and $3$ for $A=\{0,1\}$ can be found in Figure~\ref{de-Bruijn-seq-ex}. A {\em universal word}, or {\em u-word}, for $S$ is a non-circular version of a universal cycle.

\begin{figure}
\begin{center}

\vspace{-0.8cm}

\begin{tikzpicture}
\tikzset{vertex/.style = {opacity=0}}
\tikzset{edge/.style = {->,> = latex'}}

\draw (2.3,2) node{0011 =};

\draw (2,1.4) node{00};
\draw (2,1) node{01};
\draw (2,0.6) node{11};
\draw (2,0.2) node{10};

\draw (3.3,2) node{0}; 
\draw (4,1.3) node{0}; 
\draw (4.7,2) node{1};
\draw (4,2.7) node{1};

\draw (4,2) circle (0.5cm); 

\node[vertex] (a) at  (4,1.6) {};
\node[vertex] (b) at  (3.6,2.1) {}; 
\draw[edge] (a) to[bend left] (b);


\draw (7.5,2) node{00010111 =};

\draw (10,3.2) node{0};
\draw (10,0.8) node{0};
\draw (11.2,2) node{1};
\draw (8.8,2) node{0};
\draw (9.1,1.1) node{0};
\draw (9.1,2.9) node{1};
\draw (10.9,2.9) node{1};
\draw (10.9,1.1) node{1};

\draw (10,2) circle (1cm); 
\node[vertex] (c) at  (10,1.1) {};
\node[vertex] (d) at  (9.1,2.1) {}; 
\draw[edge] (c) to[bend left] (d);

\draw (7.5,1.5) node{000};
\draw (7.5,1.1) node{001};
\draw (7.5,0.7) node{010};
\draw (7.5,0.3) node{101};
\draw (7.5,-0.1) node{011};
\draw (7.5,-0.5) node{111};
\draw (7.5,-0.9) node{110};
\draw (7.5,-1.3) node{100};

\end{tikzpicture}

\end{center}
\caption{De Bruijn sequences of orders $2$ and $3$ for $A=\{0,1\}$.}\label{de-Bruijn-seq-ex}
\end{figure}
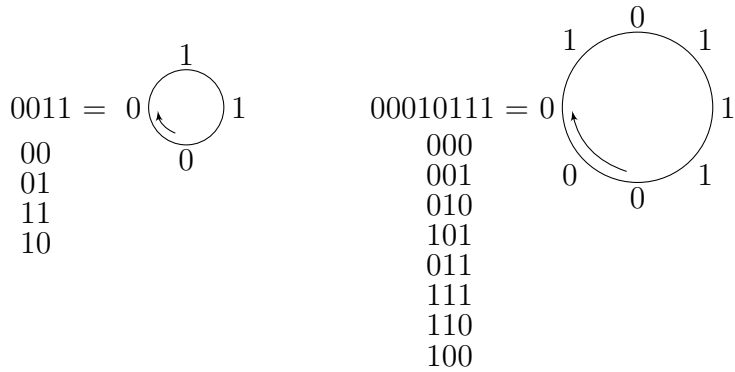

Establishing existence of u-cycles is normally done through considering {\em de  Bruijn graphs}, or similar suitable {\em transition graphs} in the context. A de Bruijn graph $B(n,k)$ consists of $k^n$ vertices corresponding to words in $A^n$ and its directed edges are $x_1x_2\cdots x_n\rightarrow x_2\cdots x_nx_{n+1}$ where $x_i\in A$ for $i\in\{1,2,\ldots,n+1\}$. See Figure~\ref{de-Bruijn-B(22)-B(23)} for $B(2,2)$ and $B(3,2)$. De Bruijn graphs are an important structure that is used in solving a variety of problems, for example, in combinatorics on words~\cite{M2005} and genomics~\cite{ZB2008}. These graphs were first introduced (for the alphabet $A=\{0,1\}$) by de Bruijn in 1944 to find the number of {\em code cycles}. 

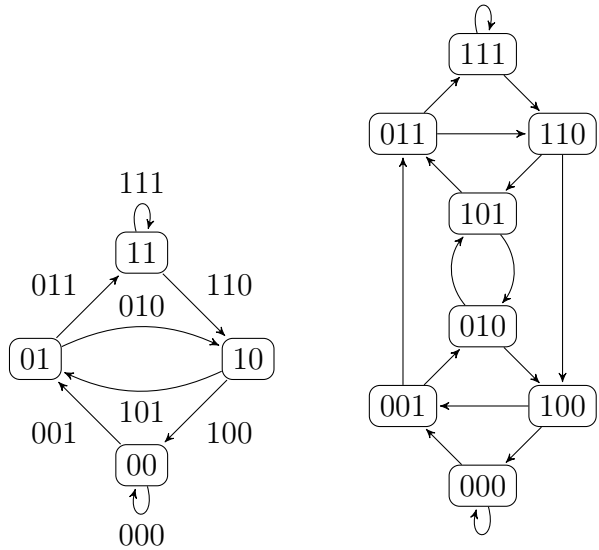
\begin{figure}
\begin{center}
\begin{tabular}{ccc}

\begin{tikzpicture}[->,>=stealth',shorten >=1pt,node distance=2cm,auto,main node/.style={rectangle,rounded corners,draw,align=center}]

\node[main node] (1) {11};
\node[main node] (2) [below left of=1] {01};
\node[main node] (3) [below right of=1] {10};
\node[main node] (4) [below right of=2] {00};

\path
(1) edge node {110} (3);
\path
(3) edge node  {100} (4)
     edge [bend left=25] node  {101} (2);
\path
(2) edge node {011} (1)
     edge [bend left=25] node  {010} (3);  
\path
(4) edge node {001} (2);         
\path
(1) edge [loop above] node {111} (1);
\path
(4) edge [loop below] node {000} (4);


\end{tikzpicture}

& 
\ \ \ \ 
&

\begin{tikzpicture}[->,>=stealth',shorten >=1pt,node distance=1.5cm,auto,main node/.style={rectangle,rounded corners,draw,align=center}]


\node[main node] (1) {111};
\node[main node] (2) [below left of=1] {011};
\node[main node] (3) [below right of=1] {110};
\node[main node] (4) [below right of=2] {101};
\node[main node] (5) [below of=4] {010};
\node[main node] (6) [below left of=5] {001};
\node[main node] (7) [below right of=5] {100};
\node[main node] (8) [below left of=7] {000};

\path
(1) edge node {} (3);
\path
(8) edge node {} (6);
\path
(4) edge node {} (2)
      edge [bend left=40] node  {} (5);
\path
(5) edge node {} (7)
     edge [bend left=40] node  {} (4);
\path
(3) edge node  {} (4)
     edge node  {} (7);
\path
(7) edge node  {} (6)
     edge node  {} (8);
\path
(6) edge node  {} (5)
     edge node  {} (2);
\path
(2) edge node  {} (1)
     edge node  {} (3);

\path
(1) edge [loop above] node {} (1);
\path
(8) edge [loop below] node {} (8);

\end{tikzpicture}

\end{tabular}

\end{center}
\caption{De Bruijn graph $B(3,2)$ is the line graph of de Bruijn graph $B(2,2)$}\label{de-Bruijn-B(22)-B(23)}
\end{figure}

\subsection{Constructing de Bruijn sequences}\label{constructing-de-Bruijn-seq}

To state some of our multi-dimensional results, we sketch here the well-known way to construct de Bruijn sequences via de Bruijn graphs, leaving justification to the reader as an easy exercise. We first need some background in graph theory.  

Let $G=(V,E)$ be a directed graph (digraph). For an edge $u\rightarrow v$ in $G$ $v$ is called the {\em head} and $u$ is called the {\em tail} of the edge. A {\em directed path} in $G$ is a sequence $v_1,\ldots,v_t$ of {\em distinct} nodes such that there is an edge $v_i\rightarrow v_{i+1}$ for each $1\leq i\leq t-1$. Such a path is  a {\em Hamiltonian path} if it contains all nodes in $G$. A closed Hamiltonian path ($v_t\rightarrow v_1$ is an edge) is a {\em Hamiltonian cycle}. If $G$ has a Hamiltonian cycle then $G$ is {\em Hamiltonian}. A digraph is {\em strongly connected} if there exists a directed path from any node to any other node.  A digraph is {\em connected} if for any pair of nodes $a$ and $b$ there exists a path in the underlying undirected graph (obtained from the digraph by removing all orientations). A {\em trail} in a digraph $G$ is a sequence $v_1,\ldots,v_t$ of nodes such that there is an edge $v_i\rightarrow v_{i+1}$ for each $1\leq i\leq t-1$ and edges are not visited more than once. An {\em Eulerian trail} in $G$ is a trail that goes through each edge exactly once. A closed Eulerian trail is an {\em Eulerian cycle}.  A directed graph is {\em Eulerian}  if it has an Eulerian cycle. Let $d^+(v)$ (resp., $d^-(v)$) denote the out-degree (resp., in-degree)  of a node $v$, which is the number of edges pointing from (resp., to) $v$. A directed graph is {\em balanced} if $d^+(v)=d^-(v)$ for each node $v$ in the graph. The following result is well-known and is not hard to prove.

\begin{thm}\label{Eulerian-trail} A digraph $G$ is Eulerian if and only if it is balanced and (strongly) connected.\end{thm}

The {\em line graph} $L(G)$ of a digraph $G$ is the digraph whose vertex set corresponds to the edge set of $G$, and $L(G)$ has an edge $e\rightarrow v$  if in $G$, the head of $e$  meets the tail of $v$. For example, the graph to the right  is the line graph of the graph to the left in Figure~\ref{de-Bruijn-B(22)-B(23)}. Clearly, an Eulerian cycle in $G$ gives a Hamiltonian cycle in $L(G)$. 

So, a traditional approach to construct de Bruijn sequences works as follows. Clearly, there is a one-to-one correspondence between directed paths in the de Bruijn graph $B(n,k)$ and words of length $\geq n$ over a $k$-letter alphabet.  For example, in Figure~\ref{corr-fig}, the path $110\rightarrow 101\rightarrow 010\rightarrow 100 \rightarrow 001$ in $B(3,2)$ corresponds to the word 1101001 since the path can be recovered from the word considering consecutive factors of length $3$. But then finding a de Bruijn sequence is equivalent to finding a Hamiltonian cycle in the respective de Bruijn graph, that is, a cycle that goes through each vertex exactly once. 

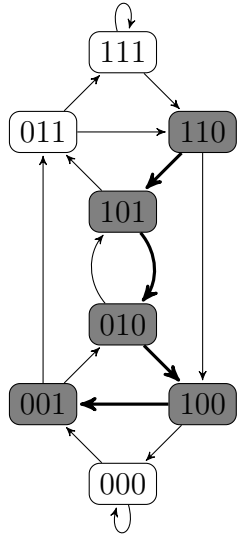
\begin{figure}
\begin{center}
\begin{tikzpicture}[->,>=stealth',shorten >=1pt,node distance=1.5cm,auto,main node/.style={rectangle,rounded corners,draw,align=center}]


\node[main node] (1) {111};
\node[main node] (2) [below left of=1] {011};
\node[main node] (3) [below right of=1, fill=gray] {110};
\node[main node] (4) [below right of=2, fill=gray] {101};
\node[main node] (5) [below of=4, fill=gray] {010};
\node[main node] (6) [below left of=5, fill=gray] {001};
\node[main node] (7) [below right of=5, fill=gray] {100};
\node[main node] (8) [below left of=7] {000};

\path
(1) edge node {} (3);
\path
(8) edge node {} (6);
\path
(4) edge node {} (2)
      edge [bend left=40, very thick] node  {} (5);
\path
(5) edge [very thick] node {} (7)
     edge [bend left=40] node  {} (4);
\path
(3) edge [very thick] node  {} (4)
     edge node  {} (7);
\path
(7) edge [very thick] node  {} (6)
     edge node  {} (8);
\path
(6) edge node  {} (5)
     edge node  {} (2);
\path
(2) edge node  {} (1)
     edge node  {} (3);

\path
(1) edge [loop above] node {} (1);
\path
(8) edge [loop below] node {} (8);

\end{tikzpicture}
\caption{A path in $B(3,2)$}\label{corr-fig}
\end{center}
\end{figure}

One can show that $B(n+1,k)$ is the line graph of $B(n,k)$, and labelling each edge $x_1x_2\cdots x_n\rightarrow x_2x_3\cdots x_{n+1}$ in $B(n,k)$ by $x_1x_2\cdots x_{n+1}$ we obtain labelling of vertices in $B(n+1,k)$ (see Figure~\ref{de-Bruijn-B(22)-B(23)} for an example). Moreover, it is easy to see that $B(n,k)$ is balanced and (strongly) connected. Hence, $B(n,k)$ is Eulerian and $B(n+1,k)$ is Hamiltonian. Constructing an Eulerian cycle in $B(n-1,k)$ using one of the known polynomial algorithms gives a Hamiltonian cycle in $B(n,k)$, which in turn gives a de Bruijn sequence.

\subsection{Martin's algorithm to generate de Bruijn sequences}\label{Martin-intro}
An alternative construction of a de Bruijn sequence is the following greedy algorithm proposed by Martin~\cite{Martin1934} in 1934.

\begin{mdframed}

\medskip
\centerline{\bf Martin's greedy algorithm to generate a de Bruijn sequence}

\medskip
Start with the word $(k-1)^{n-1}$ and then repeatedly apply the following rule: Append the {\em smallest} letter in $\{0,1,\ldots,k-1\}$ so that factors of length $n$ in the resulting word are distinct. Once no more extension is possible, remove the $n-1$ rightmost letters.
\end{mdframed}

For example, for $k=3$ and $n=2$, the steps of the algorithm are: $2\rightarrow 20 \rightarrow 200 \rightarrow 2001 \rightarrow 20010 \rightarrow 200102 \rightarrow 2001021 \rightarrow 20010211 \rightarrow 200102112 \rightarrow 2001021122\rightarrow 200102112$.

\subsection{Martin's algorithm for other combinatorial structures}\label{Martin-other-sec}

The notion of a u-cycle is well-defined for any set of combinatorial objects that admits encoding by words. U-cycles do not always exist. For example, if $k$ does not divide ${n-1\choose k-1}$ then u-cycles for the collection of $k$-sets of an $n$-set do not exist \cite{CDG1992}. A natural question is: In the case when a u-cycle exists, can Martin's algorithm (referred to as a {\em greedy algorithm} because of usage of the smallest possible option at each step) be used to construct a u-cycle? If so, which start can/cannot be used? Martin's algorithm would give an elegant, often easier justifiable way to construct u-cycles without using graph theory. For example, justifying the existence of u-cycles for permutations via Martin's algorithm in~\cite{GKSZ} is much more straightforward than the graph theoretical justification in \cite{CDG1992} (that uses a nontrivial construction in \cite{Hurlbert}); this greedy algorithm is presented in Section~\ref{greedy-perms-subsec} as it has relevance to our paper.

For a non-example, we consider {\em partitions} of the $n$-element set $\{1,2,\ldots,n\}$  discussed in  \cite{CDG1992}. Each such partition can be represented by a word over $\{1,2,\ldots\}$. For example, the word 27254552 represents the partition $\{1,3,8\}$, $\{2\}$, $\{4,6,7\}$, $\{5\}$, because in the 1st, 3rd, and 8th positions we have the same letter, so is the case with the 4th, 6th and 7th positions. It was shown in \cite{CDG1992} that u-cycles for set partitions exist for any $n\geq 4$. However, computer search shows that in the cases of $n=5$ and $n=8$ no start (of lengths 4 and 7, respectively) gives a u-cycle for set partitions if Martin's algorithm is applied. On the other hand, the greedy algorithm works for $n=4$ with the unique possible start (if the smallest possible letters are used lexicographically), namely 124: $124111121122313124$.  
For $n=6$ the algorithm works for two possible starts, 21436 and 35216, e.g.
\begin{small}
\begin{eqnarray} && 
21436111111211112211113211121211122213111213112112113212112212113312121231121232  \notag \\ && 11132211133211143212133112132112133212313232411231211231411232412131432132142231 \notag \\ && 1322421312431512231412532112432511332214335 \notag
\end{eqnarray} 
\end{small}

\noindent
gives one of the u-cycles for $n=6$, and  it does work on a unique possible start 264137 for $n=7$. Developing a theory of when the algorithm works on set partitions and when it does not, and what a possible start can be seems to be a challenging open problem. Note that in the examples we have, possible starts always consist of distinct letters. In either case, our experiments show that it is always possible to construct u-words (rather than u-cycles) for set partitions, proving which remains an open problem: 

\begin{center}
\begin{tabular}{cl}
$n=3$: & 2 possible starts; \\
$n=4$: & 6 possible starts; \\
$n=5$: & 6 possible starts; \\
$n=6$: & 48 possible starts; \\
$n=7$: & 877 possible starts.
\end{tabular}
\end{center}

The main goal of this paper is to introduce a family of universal cycles for multi-dimensional permutations via an application of a greedy algorithm. We will also discuss the graph theoretical approach to show the existence of u-cycles for multi-dimensional permutations. Before introducing multi-dimensional permutations, we note that sometimes the reduction of multi-dimensional objects to one-dimensional counterparts is straightforward. For example, the following theorem is an immediate corollary to the known results/techniques outlined in Sections~\ref{constructing-de-Bruijn-seq} and~\ref{Martin-intro}. 

\begin{thm} There exists a u-cycle for $d$-dimensional $n_1\times n_2\times \cdots\times n_d$ matrices over the alphabet $\{1,2,\ldots,k\}$. \end{thm}

\begin{proof} Think of $d$-dimensional matrices as words of length $n_d$ over the alphabet $$\{1,2,\ldots,k^{n_1\cdot n_2\cdot \ldots\cdot n_{d-1}}\},$$ where we label lexicographically all $(d-1)$-dimensional matrices. But then we deal with de Bruijn sequences whose existence is discussed in Sections~\ref{constructing-de-Bruijn-seq} and~\ref{Martin-intro}.\end{proof}

In the case of multi-dimensional permutations, reduction to known results on permutations is not so straightforward: we cannot simply refer to the existence of u-cycles theorems (and even to the theorems on products of u-cycles \cite{DiaconisGraham}). 

\subsection{$d$-dimensional permutations.}  Let $\pi = \pi_1 \pi_2 \dots \pi_n$ be a permutation of length $n$ ($n$-permutation) in the symmetric group $S_n$.  As written, $\pi$ is in one-line notation, while its two-line notation is \[\pi=\left(\begin{array}{cccc}1&2&\dots&n\\ \pi_1 & \pi_2 & \dots & \pi_n\\ \end{array}\right).\]

A \textit{$d$-dimensional permutation $\Pi$ of length $n$} (or \textit{$d$-dimensional $n$-permutation}) studied e.g.\ in \cite{AKLPT}, is an ordered $(d-1)$-tuple $(\pi^2,\pi^3, \dots , \pi^d)$ of $n$-permutations where for each $2\leq i\leq d$, $\pi^i=\pi_1^i\pi_2^i\dots\pi_n^i\in S_n$. For example, $(231,312,231)$ is a 4-dimensional permutation of length 3. We let $S^d_n$ denote the set of $d$-dimensional permutations of length $n$. Note that $S^2_n$ corresponds naturally to $S_n$, hence ``usual'' permutations are 2-dimensional permutations.  We also generalize two-line notation to $d$-line notation and we write
\renewcommand{\arraystretch}{1.25}
\[\Pi=\left(\begin{array}{cccc}1&2&\dots&n\\
\pi^2_1 & \pi^2_2 & \dots & \pi^2_n\\
 \pi^3_1 & \pi^3_2 & \dots & \pi^3_n\\
  \vdots &  & \dots & \vdots\\
   \pi^d_1 & \pi^d_2 & \dots & \pi^d_n\\ \end{array}\right)=\left(\begin{array}{cccc}\pi^1_1&\pi^1_2&\dots&\pi^1_n\\
\pi^2_1 & \pi^2_2 & \dots & \pi^2_n\\
 \pi^3_1 & \pi^3_2 & \dots & \pi^3_n\\
  \vdots &  & \dots & \vdots\\
   \pi^d_1 & \pi^d_2 & \dots & \pi^d_n\\ \end{array}\right),\]
so that $\Pi$ corresponds naturally to a $d\times n$ matrix.  It is also helpful to let $\pi^1$ denote the permutation $12\dots n$ so that we can succinctly write \[\Pi=\left\{\pi^i_j\right\}_{\begin{subarray}{l} 1 \leq i\leq d \\ 1\le j\le n\end{subarray}}.\]  Motivated by two-line notation, we say that the columns of this matrix represent the \textit{elements of $\Pi$} which we denote by  $\Pi_i$.  In particular, we write $\Pi=\Pi_1\Pi_2\dots \Pi_n$ where $\Pi_i$ is the $d$-tuple $(i,\pi^2_i, \pi^3_i,\ldots, \pi^d_i)^T$.  For example, if $\Pi = (\pi^2, \pi^3)$ is a $3$-dimensional permutation of length $5$ with $\pi^2 = 12534$ and $\pi^3 = 51243$, then we write
\renewcommand{\arraystretch}{1}
\[\Pi=\left(\begin{array}{c}\pi^1\\ \pi^2\\ \pi^3 \end{array}\right)=\left(\begin{array}{ccccc}
1&2&3&4&5\\
1&2&5&3&4\\
5&1&2&4&3\\
\end{array}\right),\]
or $\Pi=\Pi_1\Pi_2\Pi_3\Pi_4\Pi_5$ where $\Pi_1=(1,1,5)^T$, $\Pi_2=(2,2,1)^T$, $\Pi_3=(3,5,2)^T$, $\Pi_4=(4,3,4)^T$, $\Pi_5=(5,4,3)^T$. 

\subsection{A greedy algorithm to construct u-cycles for permutations}\label{greedy-perms-subsec}

For a permutation, or word, $\pi$, the {\em reduced form} of $\pi$, denoted $\red(\pi)$, is obtained by replacing the $i$-th smallest element in $\pi$ by $i$. For example, $\red(4285)=2143$. Any $i$ consecutive letters of a word or permutation $w$ form a {\em factor} of $w$. If $w$ is a cyclic word then a factor can begin at the end of $w$ and end at the beginning of $w$. If $u$ is a factor of $w$, we also say that $w$ {\em covers} $\red(u)$. 

\begin{deff} A word $U'_n$ is a {\em universal word}, or {\em u-word}, for $n$-permutations if $U'_n$ covers, in a non-cyclic way, {\em every} $n$-permutation {\em exactly once}.
\end{deff}

\begin{deff} A cyclic word $U_n$ is a {\em universal cycle}, or {\em u-cycle}, for $n$-permutations if $U_n$ covers {\em every} $n$-permutation {\em exactly once}.
\end{deff}

Let $\pi = \pi_1\pi_2\cdots \pi_m$ be a permutation of $m$ distinct integers. 
Then the {\em $i$-th extension} of $\pi$ to the right, $1\leq i\leq m$, is the permutation  
$$c_b(\pi_1)c_b(\pi_2)\cdots c_b(\pi_m)b,$$ 
where $b$ is the $i$-th smallest element in $\{\pi_1,\pi_2,\ldots,\pi_m\}$, and
$$c_b(x)=
\begin{cases}
x & \mbox{if }x<b, \\
x+1 & \mbox{if }x\geq b.
\end{cases}
$$
The {\em $(m+1)$-st extension} is the permutation $\pi b$, where $b$ is the largest element in $\{\pi_1+1,\pi_2+1,\ldots,\pi_m+1\}$. We call the first extension the {\em smallest extension}, and the $(m+1)$-st extension the {\em largest extension} of~$w$.

The following simple algorithm, suggested and justified in \cite{GKSZ}, produces a universal word $U'_n$ of length $n!+n-1$ for $n$-permutations.

\begin{mdframed}
\medskip

\centerline{\bf The greedy algorithm to construct $U'_n$}

\medskip

Begin with the permutation $U'_{n,0}:=12\cdots(n-1)$.
Suppose that a permutation
$$U'_{n,k}=a_{1}a_{2}\cdots a_{k+n-1}$$
has been constructed for $0 \leq k < n!$, and no two factors in $U'_{n,k}$ of length $n$ are order-isomorphic. Let $i$ be minimal such that no factor of length~$n$ in $U'_{n,k}$ is order-isomorphic to the $i$-th extension of $a_{k+1}a_{k+2}\cdots a_{k+n-1}$, and denote the last element of this extension by~$b$.
Then 
$$U'_{n,k+1}:=c_b(a_1)c_b(a_2)\cdots c_b(a_{k+n-1})b.$$
For some $k^*$, no extension of $U'_{n,k^*}$ will be possible without creating a factor order-isomorphic to a factor in $U'_{n,k}$. The greedy algorithm then terminates and outputs $U'_n:=U'_{n,k^*}$.
\end{mdframed}

\medskip\noindent
For example, the steps of the algorithm for $n=3$ are as follows:
$$12  \rightarrow
 231 \rightarrow
 3421 \rightarrow
  45312 \rightarrow
 564132 \rightarrow  6751324 \rightarrow
 78613245 = U'_3.$$
We note that for each $k$,  $U'_{n,k}$ is a permutation of $\{1,2,\ldots,k+n-1\}$.

The following simple extension of the greedy algorithm, again  suggested and justified in \cite{GKSZ}, turns the u-word for permutations $U'_n$ into a u-cycle for permutations $U_n$.

\begin{mdframed}

\medskip

\centerline{\bf Generating the u-cycle $U_n$ from $U'_n$}

\medskip

Remove the last $n-1$ elements in $U'_n$ and take the reduced form of the resulting sequence to obtain $U_n$.
\end{mdframed}

\medskip\noindent
For example, $U_3$ is given by
$$U'_3=78613245 \rightarrow 786132\rightarrow \red(786132)= 564132=\Pi_3.$$
For another example, $\Pi_4$ is given by
$$(22)(23)(24)(21)(20)(18)(19)3(17)42(16)1695(10)87(13)(11)(12)(15)(14).$$

\subsection{The graph theoretic approach to construct u-cycles for permutations}\label{graph-perm-u-cycles-sec}

The original approach in \cite{CDG1992} to prove the existence of u-cycles for permutations, that is similar in nature (but is much more involved) to that in Section~\ref{constructing-de-Bruijn-seq} to construct de Bruijn sequences, works as follows.  

\begin{figure}
\begin{center}

\begin{tikzpicture}[->,>=stealth',shorten >=1pt,node distance=2cm,auto,main node/.style={circle,draw,align=center}]

\node[main node] (1) {\tiny{123}};
\node[main node] (2) [below of=1] {\tiny{213}};
\node[main node] (3) [right of=1] {\tiny{132}};
\node[main node] (4) [below of=3] {\tiny{312}};
\node[main node] (5) [right of=3] {\tiny{231}};
\node[main node] (6) [below of=5] {\tiny{321}};

\path
(1) edge node {} (3)
      edge [bend left=40] node  {} (5);
\path
(2) edge node {} (1);
\path
(5) edge node {} (6);
\path
(4) edge node {} (1);
\path
(3) edge node {} (6);
\path
(6) edge node {} (4)
      edge [bend left=40] node  {} (2);

\path
(2) edge [bend left=5] node  {} (3);
\path
(3) edge [bend left=5] node  {} (2);

\path
(3) edge [bend left=10] node  {} (4);
\path
(4) edge [bend left=10] node  {} (3);

\path
(4) edge [bend left=5] node  {} (5);
\path
(5) edge [bend left=5] node  {} (4);

\path
(2) edge [bend left=5] node  {} (5);
\path
(5) edge [bend left=5] node  {} (2);

\path
(1) edge [loop above] node {} (1);
\path
(6) edge [loop below] node {} (6);

\end{tikzpicture}

\end{center}

\vspace{-0.8cm}
\caption{The graph $P(3)$}\label{graph-P3}
\end{figure}
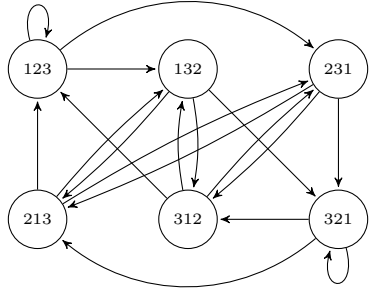

The {\em graph of overlapping permutations} $P(n)$ is defined in a way analogous to the de Bruijn graph $B(n,k)$. However, instead of requiring the tail of one permutation to equal the head of another for them to be connected by an edge, we require that the head and tail in question have their elements appear in the same relative order. Hence, the vertex set of $P(n)$ is the set of all $n!$ permutations of $\{1,2,\ldots,n\}$, and there is an edge $x_1x_2\cdots x_n\rightarrow y_1y_2\cdots y_n$ if and only if, for $2\leq i<j\leq n$, $x_i<x_j$ if and only if $y_{i-1}<y_{j-1}$. The graph $P(3)$ can be found in Figure~\ref{graph-P3}.

Via defining the notion of the {\em clustered transition graph} (where the cluster with {\em signature} $x=x_1x_2\cdots x_{n-1}$ is the set of all $n$-permutations whose first $n-1$ elements are order-isomorphic to $x$) one can prove that $P(n)$ is Hamiltonian; see  \cite{CDG1992} for details. For example, a Hamiltonian cycle $C$ for $P(3)$ is
$$132\rightarrow 312\rightarrow 123\rightarrow 231\rightarrow 321\rightarrow 213\rightarrow 132.$$

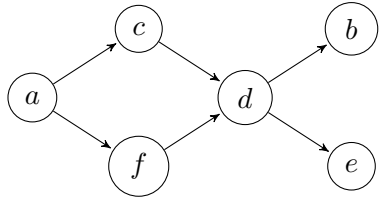
\begin{figure}
\begin{center}

\begin{tikzpicture}[->,>=stealth',shorten >=1pt,node distance=2cm,auto,main node/.style={circle,draw,align=center}]

\node[main node] (1) {{\small $a$}};
\node[main node] (2) [above right of=1,yshift=-0.5cm] {{\small $c$}};
\node[main node] (3) [below right of=1,yshift=0.5cm] {{\small $f$}};
\node[main node] (4) [below right of=2,yshift=0.5cm] {{\small $d$}};
\node[main node] (5) [above right of=4,yshift=-0.5cm] {{\small $b$}};
\node[main node] (6) [below right of=4,yshift=0.5cm] {{\small $e$}};

\path
(1) edge node {} (2);
\path
(1) edge node {} (3);
\path
(2) edge node {} (4);
\path
(3) edge node {} (4);
\path
(4) edge node {} (5);
\path
(4) edge node {} (6);

\end{tikzpicture}

\end{center}

\vspace{-0.8cm}
\caption{The relations between the elements in $U_3$, where $i\rightarrow j$ denotes the requirement that $i<j$}\label{poset-P3}
\end{figure}

Translating the Hamiltonian cycle $C$ in $P(3)$ into a u-cycle for $3$-permutations begins with assigning (as of yet) undetermined values for the potential u-cycle as $U_3=abcdef$ and building the partially ordered set (poset) in Figure~\ref{poset-P3} corresponding to the relations between the elements in $U_3$. For example, the first permutation in $C$ is 132 and hence $a<c<b$; the second permutation in $C$ is 312 and hence $c<d<b$, etc. Taking a {\em linear extension} of the poset into $\{1,2,\ldots,N\}$, that is, mapping $\{a,b,c,d,e,f\}$ into $\{1,2,\ldots,N\}$ for a suitable $N$  (in this example, $N=4$) we obtain the desired u-cycle $U_3=142342$. 

Even though obtaining a Hamiltonian cycle $C$ in $P(n)$ is straightforward, arguing that the implied ordering on the values is in fact a partial order, i.e.\ has no cycles, is a difficult task. While the authors of \cite{CDG1992}  believe that this to be the case for {\em any} Hamiltonian cycle in $P(n)$, they refer to \cite{Hurlbert} where a particular construction of $C$ is offered that guarantees the ordering  be a partial order. Since  \cite{Hurlbert} is not easily accessible, we sketch here the idea of the construction to be referred by us to in Section~\ref{last-sec}. 

First, we introduce the following notions. The {\em suffix} of a permutation is all but the first element. The {\em prefix} of a permutation is all but the last element. If we group $n$-permutations into groups whose prefixes are order-isomorphic (that is, are the same in the reduced form), then every group has a permutation ending in $i$ for each $i$. To generate a group we use the method of {\em displacement}, which is merely the sequential transposition of $i+1$ and $i$ and generates the lexicographical ordering while preserving the prefix order-isomorphism.  A {\em key} for $n$-permutations is a permutation beginning with a 1 and ending with an $n$. 

A {\em head} is a permutation beginning with a 1. The heads for a given key are generated by displacement and are hence all heads whose prefixes are order-isomorphic to the prefix of that key (we do not perform the last displacement, i.e.\ the transposition (12)).  {\em Rotations} of the heads are cyclic shifts of the heads. For example, given head $1abc$ ($n=4$) its rotations are itself, $abc1$, $bc1a$, and $c1ab$, in that order. Hence, each $n$-permutation is a rotation of one of the $(n-2)!$ heads, which in turn is the displacement of one of the $(n-2)!$ keys.  

\begin{figure}
\begin{center}
\begin{tabular}{cc}
123 & 132 \\
231 & 321 \\
312 & 213
\end{tabular}
\end{center}
\caption{Permutation list for $S_3$}\label{perm-list-S3}
\end{figure}

The only key for $S_3$ is the permutation 123, whereas for $S_4$ we have 1234 and 1324. The only other head for $S_3$ is 132, whereas 1234 includes 1243 and 1342, and 1324 includes 1423 and 1432. Figures~\ref{perm-list-S3} and~\ref{perm-list-S4} show the two lists generated so far. The lists are read by column, top to bottom, left to right.

\begin{figure}
\begin{center}
\begin{tabular}{cccccc}
1234 & 1243 & 1342 & 1324 & 1423 & 1432 \\
2341 & 2431& 3421& 3241& 4231& 4321\\
3412 & 4312 & 4213 & 2413 & 2314 & 3214 \\
4123 & 3124 & \underline{2134} & 4132 & 3142 & \underline{2143}\\
\end{tabular}
\end{center}
\caption{Temprorary list for $S_4$}\label{perm-list-S4}
\end{figure}

Notice that the heads for $S_3$ are the prefixes of the keys of $S_4$. This holds in general that the prefixes of the keys for $S_n$ are the heads for $S_{n-1}$, and in the same order. This recursive property is important in understanding the switching described in \cite{Hurlbert} and presented here properly only for $S_4$, which is sufficient for our purposes. The list for $S_4$ comes in two parts, distinguished by the two keys. Each part is indeed a cycle by itself, as is easy to see. The suffix of the last rotation of a head is precisely the prefix of the head and hence is order-isomorphic to the other heads under its key. Therefore, we obtain two cycles which we must join together. This is possible by switching the two underlined permutations $a=2134$ and $b=2143$. Letting $a'=4213$ and $b'=3214$ we see that in $P(4)$ we find the edges $a'\rightarrow b$ and $a\rightarrow b'$, so this switch does produce a Hamiltonian cycle. 

For the general switching algorithm, there are $(n-3)$ levels of joining, the success of the final one being dependent upon the construction for $S_3$; see  \cite{Hurlbert} for details, which we omit here. It can be argued that such a construction produces a partially ordered set as it has a built in {\em breaker}, the permutation $12\cdots n$, in the sense described in \cite{Hurlbert}.

For another paper involving a graph theoretic approach to construct u-cycles for permutations, see~\cite{Johnson}. There, similarly to the approach above, short cycles in $P(n)$  are created, and  a suitable approach to join these cycles is discussed that results in a u-cycle for $n$-permutations requiring $n+1$ distinct letters (the minimum possible  number of distinct letters). 

\section{Generating u-cycles for $d$-dimensional permutations}\label{u-cycle-sec}

In this section, we usually do not present the top row in multi-dimensional permutations, which is always the respective increasing permutation. Also, for a matrix $M$ in which each row has distinct elements (but there can be equal elements in different rows), the reduced form of $M$, denoted by $\red(M)$, is obtained by taking the reduced form of each row.  

\begin{deff} A matrix $U'_{d;n}$ with $d-1$ rows is a {\em universal word}, or {\em u-word}, for $d$-dimensional $n$-permutations if {\em each} $d$-dimensional permutation (without the top row) can be found non-cyclically in $U'_{d;n}$ {\em exactly once} as $n$ consecutive columns in the reduced form.
\end{deff}

For example, $U'_{3,2}=\left(\begin{array}{ccccc} 5 & 4 & 1 & 2 & 3 \\ 5 & 1 & 4 & 2 & 3 \end{array}\right)$. Indeed, the first two columns cover the permutation $\left(\begin{array}{cc} 1 & 2 \\ 2 & 1 \\ 2 & 1 \end{array}\right)$, columns 2 and 3 cover the permutation $\left(\begin{array}{cc} 1 & 2 \\ 2 & 1 \\ 1 & 2 \end{array}\right)$, and so on.

\begin{deff} Allowing in the definition of $U'_{d;n}$  consecutive columns to be considered cyclically, we define a {\em universal cycle}, or {\em u-cycle}, $U_{d;n}$ for $d$-dimensional $n$-permutations.
\end{deff}

For example, $U_{3,2}=\left(\begin{array}{cccc} 4 & 3 & 1 & 2 \\ 4 & 1 & 3 & 2  \end{array}\right)$. 

For a $d$-dimensional permutation $\Pi_1\Pi_2\dots \Pi_n$ where $\Pi_i$ is the $d$-tuple $(\pi^2_i, \pi^3_i,\ldots, \pi^d_i)^T$, we say that $\Pi_i<\Pi_j$ if $(\pi^2_i, \pi^3_i,\ldots, \pi^d_i)^T$ is lexicographically smaller than $(\pi^2_j, \pi^3_j,\ldots, \pi^d_j)^T$, that is, if $\pi^2_i<\pi^2_j$ or $\pi^2_i=\pi^2_j$ and $\pi^3_i<\pi^3_j$ or $\pi^2_i=\pi^2_j$ and $\pi^3_i=\pi^3_j$ and $\pi^4_i<\pi^4_j$, and so on.

To construct $U'_{d;n}$ and $U_{d;n}$, we mimic the steps in Section~\ref{greedy-perms-subsec} to construct  $U'_n$ and $U_n$, respectively. To proceed, we need the notion of $i$-th extension for a $d$-dimensional permutation to be introduced next. 

Let $\Pi = \Pi_1\Pi_2\cdots \Pi_m$ be a matrix with $d-1$ rows where each row contains $m$ distinct integers and columns are distinct.  Also, let $(i_1,i_2,\ldots,i_{d-1})^T$ be the $i$-th lexicographically smallest element in $\{(x_1,x_2,\ldots,x_{d-1})^T:1\leq x_i\leq m+1,1\leq i\leq d-1\}$.
Then the {\em $i$-th extension} of $\Pi$ to the right, $1\leq i\leq (m+1)^{d-1}$, is the matrix $C_B(\Pi)B$ where 
\begin{itemize}
\item $B=(b_1,b_2,\ldots,b_{d-1})^T$ and $b_j$ is the $i_j$-th smallest element in row $j$ in $\Pi$ defined to be one more than the maximum element in row $j$ if $i_j=m+1$, and
\item  $C_B(\Pi)$ is obtained from $\Pi$ by increasing by one every element in row $j$ that is $\geq b_j$ and leaving all other elements unchanged. 
\end{itemize}
We call the first extension the {\em smallest extension}, and the $(m+1)^{d-1}$-st extension the {\em largest extension} of~$\Pi$. 

For example, the second (given by $(1,1,2)^T$) and 31-st (given by $(2,4,3)^T$) extensions of $\left(\begin{array}{ccc} 4 & 2 & 5 \\ 2 & 6 & 1 \\ 4 & 1 & 2 \end{array}\right)$ are $\left(\begin{array}{cccc} 5 & 3 & 6 & 2 \\ 3 & 7 & 2 & 1 \\ 5 & 1 & 3 & 2 \end{array}\right)$ and $\left(\begin{array}{cccc} 5 & 2 & 6 & 4 \\ 2 & 6 & 1 & 7 \\ 5 & 1 & 2 & 4 \end{array}\right)$, respectively. 

In what follows, $I_{d;n}:=\left(\begin{array}{c}12\cdots n \\ \cdots \\ 12\cdots n\end{array}\right)$ is a matrix with $d-1$ rows.

\begin{mdframed}
\medskip

\centerline{\bf The greedy algorithm to construct $U'_{d;n}$}

\medskip

Begin with the matrix  $U'_{d;n,0}:=I_{d;n-1}$.
Suppose that a  $(d-1)$-dimensional permutation (with columns denoted by $\Pi_i$)
$$U'_{d;n,k}=\Pi_{1}\Pi_{2}\cdots \Pi_{k+n-1}$$
has been constructed for $0 \leq k < (n!)^{d-1}$, and no two of sets of $n$ consecutive columns (in reduced form) in $U'_{d;n,k}$ result in the same $d$-dimensional permutations. Let $i$ be minimal such that no $n$ consecutive rows in $U'_{d;n,k}$ in reduced form is the same as the $i$-th extension of $\Pi_{k+1}\Pi_{k+2}\cdots \Pi_{k+n-1}$ in reduced form, and denote the last element of this extension by~$B$.
Then 
$$U'_{d;n,k+1}:=C_B(\Pi_1\Pi_2\cdots \Pi_{k+n-1})B.$$
For some $k^*$, no extension of $U'_{d;n,k^*}$ will be possible without creating a permutation already covered by $U'_{d;n,k}$. The greedy algorithm then terminates and outputs $U'_{d;n}:=U'_{d;n,k^*}$.
\end{mdframed}

\begin{mdframed}

\medskip

\centerline{\bf Generating the u-cycle $U_{d;n}$ from $U'_{d;n}$}

\medskip

Remove the last $n-1$ elements in $U'_{d;n}$ and take the reduced form of the resulting matrix to obtain $U_{d;n}$.
\end{mdframed}

The examples of $U'_{3;2}$ and  $U_{3;2}$  above are obtained by implementing the respective algorithms. For another example, beginning with $I_{3;2}={12\choose 12}  $, we obtain the following u-cycle $U_{3;3}$:

\begin{center}
\begin{tiny}
$\left( \begin{array}{cccccccccccccccccccccccccccccccccccc}  
35 & \hspace{-2mm} 36 & \hspace{-2mm} 34 & \hspace{-2mm} 33 & \hspace{-2mm} 32 & \hspace{-2mm} 31 & \hspace{-2mm} 30 & \hspace{-2mm} 26 & \hspace{-2mm} 29 & \hspace{-2mm} 25 & \hspace{-2mm} 28 & \hspace{-2mm} 24 & \hspace{-2mm} 27 & \hspace{-2mm} 23 & \hspace{-2mm} 3 & \hspace{-2mm} 22 & \hspace{-2mm} 2 & \hspace{-2mm} 21 & \hspace{-2mm} 1 & \hspace{-2mm} 5 & \hspace{-2mm} 4 & \hspace{-2mm} 7 & \hspace{-2mm} 6 & \hspace{-2mm} 9 & \hspace{-2mm} 8 & \hspace{-2mm} 11 & \hspace{-2mm} 10 & \hspace{-2mm} 13 & \hspace{-2mm} 12 & \hspace{-2mm} 15 & \hspace{-2mm} 14 & \hspace{-2mm} 16 & \hspace{-2mm} 17 & \hspace{-2mm} 18 & \hspace{-2mm} 19 & \hspace{-2mm} 20 \\
35 & \hspace{-2mm} 36 & \hspace{-2mm} 34 & \hspace{-2mm} 2 & \hspace{-2mm} 3 & \hspace{-2mm} 1 & \hspace{-2mm} 33 & \hspace{-2mm} 32 & \hspace{-2mm} 31 & \hspace{-2mm} 4 & \hspace{-2mm} 6 & \hspace{-2mm} 5 & \hspace{-2mm} 7 & \hspace{-2mm} 9 & \hspace{-2mm} 30 & \hspace{-2mm} 8 & \hspace{-2mm} 11 & \hspace{-2mm} 10 & \hspace{-2mm} 28 & \hspace{-2mm} 29 & \hspace{-2mm} 27 & \hspace{-2mm} 26 & \hspace{-2mm} 12 & \hspace{-2mm} 14 & \hspace{-2mm} 13 & \hspace{-2mm} 16 & \hspace{-2mm} 25 & \hspace{-2mm} 15 & \hspace{-2mm} 18 & \hspace{-2mm} 17 & \hspace{-2mm} 23 & \hspace{-2mm} 24 & \hspace{-2mm} 22 & \hspace{-2mm} 19 & \hspace{-2mm} 21 & \hspace{-2mm} 20 
\end{array}
\right)$
\end{tiny}
\end{center}

We will next justify that the algorithms work. Our justification follows closely the steps in \cite{GKSZ,Martin1934}.

For $U'_{d;n,k} = \Pi_{1}\Pi_{2}\cdots \Pi_{k+n-1}$, set
$$\sigma_{k} := \red(\Pi_{k}\Pi_{k+1}\cdots \Pi_{k+n-1}), \qquad \sigma'_{k} := \red(\Pi_{k+1}\Pi_{k+2}\cdots \Pi_{k+n-1}),$$
$$J_{k} := |\{j\le k:\, \sigma'_{j} = \sigma'_{k}\}|,$$
i.e.\ $J_{k}$ is the number of occurrences in reduced form of the $d$-dimensional $(n-1)$-permutation $\sigma'_{k}$ in~$U'_{d;n,k}$.
By description of the algorithm, all $i$-th extensions of~$\sigma'_{k}$ with $i < J_k$ occur in $U'_{d;n,k}$, and the $J_k$-th extension of $\sigma'_{k}$ does not occur in~$U'_{d;n,k}$.
Therefore, the greedy algorithm terminates at~$k$ if and only if $J_k = n^{d-1}$.
If $J_k < n^{d-1}$, then $\sigma_{k+1}$ is the $J_k$-th extension of $\sigma'_{k}$.

\begin{lem}\label{lem2}
The greedy algorithm terminates at~$k$ if and only if $\sigma_{k} {\,=\,}I_{d;n}$.
\end{lem}

\begin{proof}
If $\sigma_{k} = I_{d;n}$, then $\sigma'_{k-1} = \sigma'_{k} = I_{d;n-1}$ and $J_{k-1} = n^{d-1}-1$. Thus, $J_k = n^{d-1}$, and the greedy algorithm terminates at~$k$.

Conversely, assume that $\sigma_{k} \ne  I_{d;n}$ and by the above, $\sigma_{j} \ne I_{d;n}$ for all $j \le k$. If $\sigma'_{k} \ne I_{d;n-1}$ then we have $\sigma'_{0} \ne \sigma'_{k}$, and so every occurrence of $\sigma'_{k}$ is preceeded by an element (column)  and hence $J_k \le n^{d-1}$. If $\sigma'_{k} = I_{d;n-1}$, then $\sigma'_{0} = \sigma'_{k}$ but because $\sigma_{j} = I_{d;n}$ is not possible, we have $\sigma'_j = \sigma'_k$ for at most $n^{d-1}-1$ different indices $j \ge 1$, which again gives that $J_k \leq n^{d-1}$. Therefore, the greedy algorithm does not terminate at~$k$ if $\sigma_{k} \ne I_{d;n}$ because $< n^{d-1}$ extensions of $\sigma'_{k}$ to the right are used in $U'_{d;n,k}$.
\end{proof}

\begin{lem}\label{lem3}
$U'_{n,d}$ covers all $d$-dimensional $n$-permutations.
\end{lem}

\begin{proof}
We proceed by contradiction.
Suppose that a permutation $\Pi_1 \Pi_2 \cdots \Pi_{n}$ is not covered by $U'_{d;n}$. 
Then $\Pi_2 \cdots \Pi_{n}$ is covered at most $n^{d-1}$ times, hence $\red(\Pi_2 \cdots \Pi_{n})B_n$, where $B_n=(n,n,\ldots,n)^T$ is the largest extension of $\Pi_2 \cdots \Pi_{n}$, is not covered. 
More generally, for $1 \le k \le n$, if $\red(\Pi_k \cdots \Pi_{n}) B_{n-k+2} \cdots B_n$ is not covered, then 
$$\red(\red(\Pi_{k+1} \cdots \Pi_{n}) B_{n-k+2} \cdots B_{n-1})B_n = \red(\Pi_{k+1} \cdots \Pi_{n}) B_{n-k+1} \cdots B_{n-1}B_n$$
is not covered, where, for example, $\red(\Pi_k \cdots \Pi_{n}) B_{n-k+2} \cdots B_n$ denotes $k-1$ applications of the largest extension to $\red(\Pi_k \cdots \Pi_{n})$.
We obtain that $U'_{d;n}$ does not cover the permutation $I_{d;n}$ contradicting Lemma~\ref{lem2}.
\end{proof}

By the nature of the greedy algorithm, $U'_{d;n}$ cannot cover a permutation more than once. By Lemma~\ref{lem3}, $U'_{d;n}$ covers all $n$-permutations. Hence, we proved the following result.

\begin{thm}\label{u-word-thm} 
$U'_{d;n}$ is a u-word for $d$-dimensional $n$-permutations.
\end{thm}


\begin{thm}\label{u-cylce-d-dim-thm}
 $U_{d;n}$ is a u-cycle for $d$-dimensional $n$-permutations.
\end{thm}

\begin{proof}
By Theorem~\ref{u-word-thm}, it suffices to prove that 
$$\red(\Pi_{k}\Pi_{k+1} \cdots \Pi_{k+n-1}) = \red(\Pi_{k}\cdots \Pi_{(n!)^{d-1}} \Pi_1 \cdots \Pi_{k+n-(n!)^{d-1}-1})$$ 
for all $(n!)^{d-1}-n+2 \le k \le (n!)^{d-1}$. Note that $\sigma_i$ ends with $(1,1,\ldots,1)^T$ for all $1\leq i < n$. 
Hence, letting as above $B_x$ denote the column $(x,x,\ldots,x)^T$, we have that
$$U'_{d;n,n-1} = B_nB_{n+1}\cdots B_{2n-2}B_{n-1}B_{n-2}\cdots B_1$$
and so, for all $k \ge n$,
$$\Pi_k < \Pi_1 < \Pi_2 < \cdots < \Pi_{n-1}.$$

Next, we show that $\sigma_i$ ends with $B_n$ for $(n!)^{d-1}-n+2\leq i\leq (n!)^{d-1}$.
Suppose that this is not true for some such $i$. Since $U'_{d;n}$ is a u-word, we must have $\sigma'_{i-1} = \sigma'_{j-1}$ for some $j > i$. 
It follows from $\sigma_{(n!)^{d-1}} = I_{d;n}$ that 
$$\Pi_{(n!)^{d-1}} < \Pi_{(n!)^{d-1}+1} < \cdots < \Pi_{j+n-2}.$$
Then, $\sigma'_{i-1} = \sigma'_{j-1}$ implies that 
$$\Pi_{(n!)^{d-1}+i-j} < \Pi_{(n!)^{d-1}+i-j+1} < \cdots < \Pi_{i+n-2}< \cdots < \Pi_{(n!)^{d-1}+n-1}.$$
Iterating this argument gives that $\Pi_{j} < \Pi_{j+1} < \cdots < \Pi_{(n!)^{d-1}+n-1}$ and thus $\Pi_{i} < \Pi_{i+1} < \cdots < \Pi_{(n!)^{d-1}+n-1}$, contradicting the assumption that $\sigma_i$ does not end with~$B_n$.
Therefore, $\sigma_i$ ends with $B_n$ for all $i \ge (n!)^{d-1}-n+2$, and hence
$$\Pi_k < \Pi_{(n!)^{d-1}+1} < \Pi_{(n!)^{d-1}+2} < \cdots < \Pi_{(n!)^{d-1}+n-1}$$
for all $(n!)^{d-1}-n+2 \le k \le n!$. Therefore,
$$\sigma_{k} = \red(\Pi_{k}\Pi_{k+1} \cdots \Pi_{(n!)^{d-1}}) B_{(n!)^{d-1}-k+2} \cdots B_n = \red(\Pi_{k}\cdots \Pi_{(n!)^{d-1}} \Pi_1 \cdots \Pi_{k+n-(n!)^{d-1}-1})$$ 
for $(n!)^{d-1}-n+2 \le k \le n!$, as desired. 
\end{proof}

\begin{rem}\label{family-of-cycles} In fact, Theorem~\ref{u-cylce-d-dim-thm} gives us a family of u-cycles for multi-dimensional permutations.  Indeed,  it is easy to see that {\em complementing} any row, that is, swapping the smallest letter with the largest letter, the next smallest letter with the next largest letter, etc, in any given row, results in a u-cycle for multi-dimensional permutations, and we have $2^{d-1}$ different u-cycles produced from a given initial u-cycle in this way (the beginnings of length $n-1$ will be distinct). On the other hand, note that such operations as reversing a row or taking its cyclic shift potentially can result in a non-u-cycle for multi-dimensional permutations.\end{rem}

To illustrate Remark~\ref{family-of-cycles}, complementing the second row of $U_{3;3}$ given above, we obtain

\begin{center}
\begin{tiny}
$\left( \begin{array}{cccccccccccccccccccccccccccccccccccc}  
35 & \hspace{-2mm} 36 & \hspace{-2mm} 34 & \hspace{-2mm} 33 & \hspace{-2mm} 32 & \hspace{-2mm} 31 & \hspace{-2mm} 30 & \hspace{-2mm} 26 & \hspace{-2mm} 29 & \hspace{-2mm} 25 & \hspace{-2mm} 28 & \hspace{-2mm} 24 & \hspace{-2mm} 27 & \hspace{-2mm} 23 & \hspace{-2mm} 3 & \hspace{-2mm} 22 & \hspace{-2mm} 2 & \hspace{-2mm} 21 & \hspace{-2mm} 1 & \hspace{-2mm} 5 & \hspace{-2mm} 4 & \hspace{-2mm} 7 & \hspace{-2mm} 6 & \hspace{-2mm} 9 & \hspace{-2mm} 8 & \hspace{-2mm} 11 & \hspace{-2mm} 10 & \hspace{-2mm} 13 & \hspace{-2mm} 12 & \hspace{-2mm} 15 & \hspace{-2mm} 14 & \hspace{-2mm} 16 & \hspace{-2mm} 17 & \hspace{-2mm} 18 & \hspace{-2mm} 19 & \hspace{-2mm} 20 \\

2& \hspace{-2mm}  1& \hspace{-2mm}  3& \hspace{-2mm}  35& \hspace{-2mm}  34& \hspace{-2mm}  36& \hspace{-2mm}  4& \hspace{-2mm}  5& \hspace{-2mm}  6& \hspace{-2mm}  33& \hspace{-2mm}  31& \hspace{-2mm}  32& \hspace{-2mm}  30& \hspace{-2mm}  28& \hspace{-2mm}  7& \hspace{-2mm}  29& \hspace{-2mm}  26& \hspace{-2mm}  27& \hspace{-2mm}  9& \hspace{-2mm}  8& \hspace{-2mm}  10& \hspace{-2mm}  11& \hspace{-2mm}  25& \hspace{-2mm}  23& \hspace{-2mm}  24& \hspace{-2mm}  21& \hspace{-2mm}  12& \hspace{-2mm}  22& \hspace{-2mm}  19& \hspace{-2mm}  20& \hspace{-2mm}  14& \hspace{-2mm}  13& \hspace{-2mm}  15& \hspace{-2mm}  18& \hspace{-2mm}  16& \hspace{-2mm}  17

\end{array}
\right)$
\end{tiny}
\end{center}

\section{Concluding remarks}\label{last-sec}

This paper shows that Martin's greedy algorithm can be used to construct u-cycles for multi-dimensional permutations, while as is discussed in Section~\ref{Martin-other-sec}, this algorithm is not applicable for some sets of combinatorial objects. That would be interesting to build a theory of applicability/suitability of Martin's greedy algorithm in the situations when u-cycles exist. In particular, understanding when the algorithm works on set partitions and when it does not, and what the possible start can be, is an interesting open problem. 

In fact, the notion of a greedy algorithm can be modified, for example, by alternating steps of taking the smallest and largest available options. Such an approach does not work, for example, for the set of all words of length 2 over $\{0,1,2\}$ if we start with 2 (analogously to the original Martin's algorithm) as we get stack before covering all words: $2\rightarrow 20\rightarrow 202 \rightarrow 2021 \rightarrow 20212 \rightarrow 202122$. But are there situations when such a modification would work? Or can it be proved that it never works?

The greedy algorithm to generate u-cycles for multi-dimensional permutation is rather elegant, but can one use the typical graph theoretic approach, analogues to that in Section~\ref{graph-perm-u-cycles-sec} for permutations, to achieve the same goal? At least constructing u-words for multi-dimensional permutations is straightforward here as one can easily generalize the notion of $P(n)$ to that of $P_d(n)$ for $d$-dimensional $n$-permutations, where the vertices are all such permutations and there is an edge from a permutation $X$ to a permutation $Y$ if the suffix of each row of $X$ is order-isomorphic to the prefix of the same row of $Y$.   One can then mimic the steps of clustering of $P(n)$ in \cite{CDG1992} to get clustering of $P_d(n)$ (where cluster's in- and out-degrees will be $n^{d-1}$) and then deduce Hamiltonicity of $P_d(n)$ that can be easily translated into a u-word. 

\begin{figure}[ht]
\begin{center}
\begin{tabular}{cc}
$\left(\begin{array}{ccc} 1 & 2 & 3 \\ 1 & 3 & 2 \end{array}\right)$ & $\left(\begin{array}{ccc} 1 & 3 & 2 \\ 1 & 2 & 3 \end{array}\right)$ \\[4mm]
$\left(\begin{array}{ccc} 2 & 3 & 1 \\ 3 & 2 & 1 \end{array}\right)$ & $\left(\begin{array}{ccc} 3 & 2 & 1 \\ 2 & 3 & 1 \end{array}\right)$ \\[4mm]
$\left(\begin{array}{ccc} 3 & 1 & 2 \\ 2 & 1 & 3 \end{array}\right)$ & $\left(\begin{array}{ccc} 2 & 1 & 3 \\ 3 & 1 & 2 \end{array}\right)$ 
\end{tabular}
\end{center}
\caption{A key group and the respective head group producing a small cycle read by column, top to bottom, left to right}\label{123-132}
\end{figure}

However, proving that there is a u-cycle using the methods in \cite{Hurlbert} outlined in Section~\ref{Martin-other-sec} does not seem to work at least for $n=3$. Indeed, consider the case of $d=3$. It is natural (but maybe not the only way?) to define the notion of key and head requiring the second row (out of three, the top row is always 123 and hence can be removed) to be key and head, respectively (recall the definitions in Section~\ref{Martin-other-sec}). Hence, instead of 2 groups of 3-permutations in Figure~\ref{perm-list-S3}, we will have 12 groups of 3 3-permutations as each group in Figure~\ref{perm-list-S3} will be repeated 6 times (the number of ways to chose the second independent permutation).  We would expect a key group and the head group corresponding to it to be connected into a small cycle, like it is the case  in Figure~\ref{perm-list-S3}, with a hope to connect the cycles between each other by the switching, or other methods. This is indeed the case for the key group generated by $\left(\begin{array}{ccc} 1 & 2 & 3 \\ 1 & 2 & 3 \end{array}\right)$ and the respective head group (as it is not different from the $d=2$ case), and even for the key group generated by $\left(\begin{array}{ccc} 1 & 2 & 3 \\ 1 & 3 & 2 \end{array}\right)$ and the respective head group (see Figure~\ref{123-132}). However, this does not work, for example, for the key group generated by $\left(\begin{array}{ccc} 1 & 2 & 3 \\ 2 & 3 & 1 \end{array}\right)$ and the respective head group (see Figure~\ref{123-231}).

\begin{figure}[h]
\begin{center}
\begin{tabular}{cc}
$\left(\begin{array}{ccc} 1 & 2 & 3 \\ 2 & 3 & 1 \end{array}\right)$ & $\left(\begin{array}{ccc} 1 & 3 & 2 \\ 2 & 1 & 3 \end{array}\right)$ \\[4mm]
$\left(\begin{array}{ccc} 2 & 3 & 1 \\ 3 & 1 & 2 \end{array}\right)$ & $\left(\begin{array}{ccc} 3 & 2 & 1 \\ 1 & 3 & 2 \end{array}\right)$ \\[4mm]
$\left(\begin{array}{ccc} 3 & 1 & 2 \\ 1 & 2 & 3 \end{array}\right)$ & $\left(\begin{array}{ccc} 2 & 1 & 3 \\ 3 & 2 & 1 \end{array}\right)$ 
\end{tabular}
\end{center}
\caption{A key group and the respective head group {\em not} producing a small cycle read by column, top to bottom, left to right}\label{123-231}
\end{figure}

We  leave it as an open question to modify the method in \cite{Hurlbert} to construct u-cycles for permutations (sketched in Section~\ref{graph-perm-u-cycles-sec}) to obtain a u-cycle for multi-dimensional permutations. Analogously to the conjecture about $P(n)$, it is conceivable that {\em any} Hamiltonian cycle in $P_d(n)$ can be turned into a u-cycle, but this may be difficult to confirm.  One can also study the minimum number of letters required to produce a u-cycle for multi-dimensional permutations, as it is done in \cite{Johnson} for usual permutations. However, the steps in \cite{Johnson} do not seem to be extendable (easily) to the case of multi-dimensional permutations. 

\vskip 3mm
\noindent {\bf Acknowledgments.} The authors are grateful to Glenn Hulbert for sharing with us a copy of  \cite{Hurlbert}. The first  author is supported by Leverhulme Research Fellowship (grant reference RF-2023-065\textbackslash 9).  The second author is supported by the National Natural Science Foundation of China (NSFC) grants 12171034 and 12271023.

\end{document}